\documentclass[11pt]{amsart} 
\usepackage{amscd,amssymb,amsxtra}
\usepackage[mathscr]{eucal}% \usepackage{mathrsfs}  % \mathscr{A}
\usepackage{mathabx}
\usepackage{comment}
\usepackage{color}
\usepackage{enumitem} 
\makeatletter
\let\@secnumfont\bfseries
\def\section{\@startsection{section}{1}%
  \z@{4\linespacing\@plus\linespacing}{\linespacing}%
  {\bfseries\centering}}
\def\introsection{\@startsection{section}{1}%
  \z@{3\linespacing\@plus\linespacing}{\linespacing}%
  {\bfseries\centering}}
\def\subsection{\@startsection{subsection}{2}%
   \z@{1.25\linespacing\@plus.7\linespacing}{.5\linespacing}%
   {\normalfont\bfseries}}
\def\subsectionsinline{\def\subsection{\@startsection{subsection}{2}%
  \z@{1\linespacing\@plus.7\linespacing}{-.5em}%
  {\normalfont\bfseries}}}

\makeatother

\theoremstyle{definition}

\newtheorem*{definition*}{Definition}
\newtheorem*{example*}{Example}
\newtheorem*{problem*}{Problem}
\newtheorem*{exercise*}{Exercise}
\newtheorem*{question*}{\color{blue}Question}
\newtheorem*{construction*}{Construction}

\theoremstyle{remark}

\newtheorem{remark}[equation]{Remark}

\newtheorem*{note*}{Note}
\newtheorem*{notation*}{Notation}
\newtheorem*{remark*}{Remark}
\newtheorem*{data*}{Data}

\theoremstyle{plain}
\newtheorem{theorem}[equation]{Theorem}

\newtheorem{lemma}[equation]{Lemma}
\newtheorem{proposition}[equation]{Proposition}

\newtheorem*{theorem*}{Theorem}
\newtheorem*{corollary*}{Corollary}
\newtheorem*{lemma*}{Lemma}
\newtheorem*{proposition*}{Proposition}
\newtheorem*{conjecture*}{Conjecture}
\newtheorem*{claim*}{Claim}
\newtheorem*{proposal*}{Proposal}
\newtheorem*{conclusion*}{Conclusion}
\newtheorem*{hypothesis*}{Hypothesis}
\newtheorem*{assumption*}{Assumption}

\numberwithin{equation}{section}

\definecolor{refkey}{rgb}{0,.6,.4}

\renewcommand{\:}{\colon}
\newcommand{\Ahat}{{\hat A}}

\DeclareMathOperator{\Hom}{Hom}

\DeclareMathOperator{\pt}{pt}

\newcommand{\RR}{{\mathbb R}}

\DeclareMathOperator{\tr}{tr}

\newcommand{\chiup}{\raise.5ex\hbox{$\chi$}}

\DeclareRobustCommand{\mstrut}{^{\vphantom{1*\prime y\vee M}}}

\newcommand{\temsquare}{\raise3.5pt\hbox{\boxed{ }}}

\usepackage[all,2cell]{xy}
\usepackage[colorlinks,citecolor=refkey]{hyperref}

\definecolor{refkey}{rgb}{0,.8,.2}\definecolor{labelkey}{rgb}{1,0,0} 

\makeatletter
\newcommand{\raisemath}[1]{\mathpalette{\raisem@th{#1}}}
\newcommand{\raisem@th}[3]{\raisebox{#1}{$#2#3$}}
\makeatother

\DeclareMathOperator{\Det}{Det}
\DeclareMathOperator{\Sym}{Sym}
\newcommand{\BNG}{B^{}_{\nabla }G}

\newcommand{\ENG}{E^{}_{\nabla }G}
\newcommand{\Man}{\mathbf{Man}}
\newcommand{\PG}{(P_G)_\nabla }
\newcommand{\PQ}{P_Q}
\newcommand{\SG}{(S_G)_\nabla }
\newcommand{\TG}{\Theta \mstrut _G}
\newcommand{\XG}{(X_G)_\nabla }
\newcommand{\XQ}{X_Q}
\newcommand{\SQ}{S_Q}
\newcommand{\comp}[2]{\Pi _{\raisemath{-2pt}{#1}}^{\raisemath{2pt}{#2}}}
\newcommand{\hx}{\tilde\xi }
\newcommand{\sF}{\mathcal{F}}

\begin{document}

\abovedisplayskip18pt plus4.5pt minus9pt
\belowdisplayskip \abovedisplayskip
\abovedisplayshortskip0pt plus4.5pt
\belowdisplayshortskip10.5pt plus4.5pt minus6pt
\baselineskip=15 truept
\marginparwidth=55pt

\makeatletter
\renewcommand{\tocsection}[3]{%
  \indentlabel{\@ifempty{#2}{\hskip1.5em}{\ignorespaces#1 #2.\;\;}}#3}
\renewcommand{\tocsubsection}[3]{%
  \indentlabel{\@ifempty{#2}{\hskip 2.5em}{\hskip 2.5em\ignorespaces#1%
    #2.\;\;}}#3} 
\makeatother

\setcounter{tocdepth}{2}

\renewcommand{\theequation}{\arabic{equation}}   % omit if frenchstyle

%**end of header

% lasteq@ 16
% lastsec@  1
% lastthm@  6
% lastfig@  1

 \title{On equivariant Chern-Weil forms and determinant lines} 
 \author[D. S. Freed]{Daniel S.~Freed}
 \thanks{The work of D.S.F. is supported by the National Science Foundation
under grant DMS-1207817.} 
 \address{Department of Mathematics \\ University of Texas \\ Austin, TX
78712} 
 \email{dafr@math.utexas.edu}
 \date{\today}
 \begin{abstract} 
 A strong from of invariance under a group~$G$ is manifested in a family over
the classifying space~$BG$.  We advocate a differential-geometric avatar
of~$BG$ when $G$~is a Lie group.  Applied to $G$-equivariant connections on
smooth principal or vector bundles, the
$\emph{equivariance}\to\emph{families}$ principle converts the
$G$-equivariant extensions of curvature and Chern-Weil forms to the standard
nonequivariant versions.  An application of this technique yields the
moment map of the determinant line of a $G$-equivariant Dirac operator, which
in turn sheds light on some anomaly formulas in quantum field theory.
 \end{abstract}
\maketitle

%\pagestyle{myheadings}   % omit in final
%\markboth{PRELIMINARY VERSION (\today)}{PRELIMINARY VERSION (\today)}  % omit

%{\small\tableofcontents}

%{\small
%\def\reftext{References}
%\renewcommand{\tocsection}[3]{%
%  \begingroup 
%   \def\tmp{#3}% 
%   \ifx\tmp\reftext
%  \indentlabel{\phantom{1}\;\;} #3%
%  \else\indentlabel{\ignorespaces#1 #2.\;\;}#3%
%  \fi\endgroup}
%\tableofcontents}

Let $X$~be a smooth manifold, $H$~a Lie group, and $P\to X$ a smooth
principal $H$-bundle with connection~$\Theta $.  Suppose that a Lie group~$G$
acts on $P\to X$ preserving~$\Theta $.  If $G$~acts freely, so that there is
a quotient principal $H$-bundle $P/G\to X/G$, then there is an obstruction to
descending~$\Theta $ to the quotient bundle: the moment map.  This
obstruction---defined for not-necessarily-free $G$-actions---is the key
ingredient in a $G$-equivariant extension of the curvature of~$\Theta $, so
too in the $G$-equivariant extension of the Chern-Weil forms~\cite{BV}, which
live in $G$-equivariant de Rham theory.
 
Especially in topology and algebraic geometry a strong form of $G$-invariance
is expressed by fibering over a classifying space~$BG$.  The particular form
of the classifying space varies with context.  Here we advocate in
differential geometry for~$\BNG$, the classifying ``generalized manifold'' of
$G$-connections introduced and studied in~\cite{FH}.  As we review below,
$\BNG$~is a simplicial sheaf on the site of smooth manifolds, a generalized
manifold in the same sense that a distribution is a generalized function.  A
$G$-manifold $X$ has a \emph{differential Borel quotient}, which is a fiber
bundle $\XG\to\BNG$ with fiber~$X$, and we proved~\cite[Theorem~7.28(ii)]{FH}
that the de Rham complex of~$\XG$ is the Weil model for $G$-equivariant de
Rham theory.  Here, given a $G$-equivariant connection~$\Theta $ on $P\to X$,
the strong form of $G$-invariance is a connection~$\TG$ on the differential
Borel quotient $\PG\to\XG$.  Our main theorem identifies the curvature
of~$\TG$ with the $G$-equivariant curvature of~$\Theta $, and similarly the
Chern-Weil forms of~$\TG$ with the $G$-equivariant Chern-Weil forms
of~$\Theta $.  (A less rigid version of this construction was used
in~\cite{BT} to prove that equivariant Chern-Weil forms represent equivariant
characteristic classes.)

The differential Borel quotient converts equivariance into a fiber bundle.
As an application of this technique we prove a theorem about $G$-equivariant
families of Dirac operators.  In~\cite{BF} we constructed a \emph{canonical}
connection on the determinant line bundle of a family and computed a formula
for its curvature.  Attached to a $G$-equivariant family of Dirac operators
we obtain a $G$-equivariant determinant line bundle, and here we compute the
corresponding moment map.  The proof uses the
$\emph{equivariance}\to\emph{families}$ construction to reduce the moment map
computation to the known curvature formula.  Applied to quantum field theory
we find a direct geometric interpretation of standard ``covariant'' anomaly
formulas in the physics literature (e.g.,~\cite{BZ,ASZ,AgG,AgW}).

  \subsection*{Equivariant connections and $\BNG$}

The main construction and computation lie squarely in differential geometry:
no simplicial sheaves required.  Let $\pi \:P\to X$ be a fiber bundle and $W$
a horizontal distribution: $W\subset TP$~is a subbundle and the inclusion
maps induce an isomorphism $W\oplus T(P/X)\cong TP$, where $T(P/X)=\ker \pi
_*$ is the subbundle of vertical vectors.  If $\pi $~is a principal
$H$-bundle for a Lie group~$H$ and $W$~is $H$-invariant, then $W$~is a
connection, but our construction is more general (and we need the general
version in the next section).  Let $G$~be a Lie group which acts on $\pi
\:P\to X$ and preserves the distribution~$W$.  Suppose $Q\to M$ is a
principal $G$-bundle with connection~$V\subset TQ$ a $G$-invariant horizontal
distribution.  Then the distribution $V\oplus W\subset T(Q\times P)$ is
$G$-invariant, so descends to a horizontal distribution~$W_Q$ on the quotient
fiber bundle~$\pi _Q$ in the iterated fibration
  \begin{equation}\label{eq:1}
\begin{gathered}
     \xymatrix@R=4pt{Q\times \mstrut _GP\ar[r]^{\pi _Q\;}\ar@{=}[d]& Q\times
     \mstrut _GX 
     \ar[r]^<<<<{\rho}\ar@{=}[d]& M \\ 
     \PQ&\XQ}
     \end{gathered}
  \end{equation}
The restriction of~$W_Q$ over a fiber of $\rho $ may be identified with~$W$.
The horizontal distribution~$W_Q$ is functorial for maps of $G$-bundles with
connection.
 
In case $\pi \:P\to X$ is a principal $H$-bundle and $W$~a connection, we
compute the curvature of~$W_Q$.  Let $\Theta \in \Omega ^1_P(\mathfrak{h})$
be the connection form with kernel~$W$.  Assume for simplicity that $G$~is
finite dimensional, and fix a basis $\{e_a\}\subset \mathfrak{g}$.  Define
the structure constants~$f_{ab}^c\in\RR$ by $[e_a,e_b]=f_{ab}^ce_c$.  Each
basis element~$e_a$ induces a vector field on~$P$, via the infinitesimal
$G$-action, and so a contraction operator $\iota _a\:\Omega ^{\bullet
}_P\to\Omega ^{\bullet -1}_P$ of degree~$-1$ on differential forms (with
coefficients).  Let $\phi =\phi ^ae_a\in \Omega ^1_Q(\mathfrak{g})$ be the
connection form with kernel~$V$.  Let $\Omega =d\Theta +\frac 12[\Theta
\wedge \Theta ]\in \Omega ^2_P(\mathfrak{h})$ and $\omega =d\phi +\frac
12[\phi \wedge \phi ]\in \Omega ^2_Q(\mathfrak{g})$ be the curvatures
of~$\Theta $ and~$\phi $, respectively.  Write $\omega =\omega ^ae_a$.
Finally, let $\Theta _Q$, $\Omega _Q$ denote the connection and curvature of
the connection~$W_Q$ on the principal $H$-bundle $\pi \mstrut _Q$
in~\eqref{eq:1}.

  \begin{proposition}[]\label{thm:1}
 \ 
 \begin{enumerate}[label=\textnormal{(\roman*)}]

 \item $\Theta _Q=\Theta -\phi ^a\cdot\iota _a\Theta $.

 \item $\Omega _Q=\,\Omega -\phi ^a\wedge \iota _a\Omega + \frac 12\phi
^a\wedge \phi ^b\cdot\iota _b\iota _a\Omega - \omega ^a\cdot\iota _a\Theta $.

 \end{enumerate} 
  \end{proposition}

  \begin{proof}
 The form on the right hand side of~(i) vanishes on~$V\oplus W$ and is the
identity on vertical vectors, since $\Theta $~is.  For~(ii) we compute using
$(d\iota _a+\iota _ad)\Theta =0$ by $G$-invariance:
  \begin{equation}\label{eq:2}
     \Omega _Q = \Omega \;-\;d(\phi ^a\cdot \iota _a\Theta ) \;-\;[\phi 
     ^a\cdot \iota _a\Theta \wedge \Theta] \;+\;\frac 12[\phi ^a\cdot \iota
     _a\Theta \wedge \phi ^b\cdot \iota _b\Theta ] 
  \end{equation}
and  
  \begin{equation}\label{eq:3}
     -d(\phi ^a\cdot \iota _a\Theta ) = -\omega ^a\wedge \iota _a\Theta
     \;+\;\frac 
     12f^c_{ab}\,\phi ^a\wedge \phi ^b\cdot \iota _c\Theta \;-\;\phi ^a\wedge
     \iota _a\Omega \;+\; \phi ^a\wedge [\iota _a\Theta , \Theta ]. 
  \end{equation}
The $G$-invariance of~$\Theta $ also implies 
  \begin{equation}\label{eq:4}
     \iota _b\iota _a\Omega =f_{ab}^c\,\iota _c\Theta \;+\;[\iota _a\Theta
     ,\iota _b\Theta ]. 
  \end{equation}
Combine~\eqref{eq:2}, \eqref{eq:3}, and the product of~\eqref{eq:4} with
$\frac 12\phi ^a\wedge \phi ^b$ to conclude.
  \end{proof}

Let $\Man$~denote the category of smooth finite dimensional manifolds and
smooth maps.  A \emph{simplicial presheaf} is a contravariant functor~$\sF$
from~$\Man$ to the category of simplicial sets.  It is a simplicial
\emph{sheaf} if it satisfies a covering condition; see~\cite{FH} and the
references therein.  An object~$M\in \Man$ is a ``test manifold'', and
$\sF(M)$~is the value of the sheaf~$\sF$ on that test manifold.  The sheaves
we consider have values in the category of groupoids: an object in~ $\BNG(M)$
is a principal $G$-bundle $Q\to M$ with connection~$\phi $ and a morphism is
an isomorphism of principal $G$-bundles which preserves the connections.
There is a universal $G$-bundle $\ENG\to\BNG$; an object in~$\ENG(M)$ is a
principal $G$-bundle $Q\to M$ with connection and a section.  A smooth finite
dimensional manifold~$X$ is a sheaf via the Yoneda embedding.  Yoneda also
implies that $\sF(M)$~is the set of maps $M\to\sF$ in the category of
sheaves.  If $X$~carries a left $G$-action, then we form the
\emph{differential Borel construction}
  \begin{equation}\label{eq:5}
     \XG := \ENG\times \mstrut _GX\xrightarrow{\;\;\rho \;\;} \BNG, 
  \end{equation}
a fiber bundle with fiber~$X$.  Therefore, given $\pi \:P\to X$ which is
$G$-equivariant, then \eqref{eq:1} is the pullback of the iterated fiber bundle 
  \begin{equation}\label{eq:6}
     \PG\xrightarrow{\;\;\pi _G\;\;} \XG \xrightarrow{\;\;\rho \;\;} \BNG
  \end{equation}
via the map $M\to\BNG$ represented by the principal $G$-bundle $Q\to M$ with
connection~$\phi $.  If $W$~is a $G$-invariant horizontal distribution on
$\pi $, then the construction at the beginning of this section gives a
horizontal distribution~$W_G$ on~$\pi _G$.  It encodes the ``strong
$G$-invariance'' of~$W$. 
 
The de Rham complex of~$\ENG$ is the \emph{Weil algebra} $\Sym ^{\bullet
}\mathfrak{g}^*\otimes {\textstyle\bigwedge} ^{\bullet }\mathfrak{g}^*$,
which is the free differential graded algebra
on~$\mathfrak{g}^*={\textstyle\bigwedge} ^1\mathfrak{g}^*$.  Let $\{\theta
^a\}\subset {\textstyle\bigwedge} ^1\mathfrak{g}^*$ be the dual basis to
$\{e_a\}\subset \mathfrak{g}$; then $\{\chi ^a=d\theta ^a\}\subset
\Sym^1\mathfrak{g}^*$ is also a basis.  Set $\nu ^a = d\theta ^a + \frac
12f^a_{bc}\theta ^b\theta ^c$.  If $X$~is a smooth $G$-manifold, then the
\emph{Weil model} of equivariant de Rham theory is the basic subcomplex of
$\Sym^{\bullet }\mathfrak{g}^*\otimes {\textstyle\bigwedge} ^{\bullet
}\mathfrak{g}^*\otimes \Omega ^{\bullet }_X$; see ~\cite[\S5]{MQ}, \cite{GS}.
Following H.~Cartan~\cite{C1,C2}, these references also construct a
quasi-isomorphism with the \emph{Cartan model}, the $G$-invariant subcomplex
of $\Sym^{\bullet }\mathfrak{g}^*\otimes \Omega ^{\bullet }_X$ with
differential $d_X - \iota _\xi =d_X - \chi ^a\iota _a$, where $\xi $~is the
$G$-invariant $\mathfrak{g}^*$-valued vector field on~$X$ which expresses the
infinitesimal $G$-action.  The quasi-isomorphism is the augmentation map of
the exterior algebra: it sends $\theta ^a\to0$, $\nu ^a\to\chi ^a$ for
all~$a$.  As quoted earlier from~\cite{FH}, the de Rham complex of~$\XG$ is
the Weil model of equivariant de Rham theory on~$X$.

  \begin{theorem}[]\label{thm:2}
 Let $\pi \:P\to X$ be a principal $H$-bundle with connection~$\Theta $, and
suppose $G$~acts on~$P\to X$ preserving~$\Theta $.  Then the curvature of the
induced connection on $\pi _G\:\PG\to\XG$ is 
  \begin{equation}\label{eq:7}
     \Omega _G=\Omega -\chi ^a\cdot\iota _a\Theta 
  \end{equation}
in the Cartan model.   
  \end{theorem}

\noindent
 Let $\mathfrak{g}\mstrut _P\to X$ be the adjoint bundle of Lie algebras;
then $(e_a\mapsto\iota _a\Theta )\in \Hom\bigl(\mathfrak{g},\Omega
^0_X(\mathfrak{g}\mstrut _P)\bigr)$ is the moment map.  $\Omega _G$~is the
$G$-equivariant extension of the curvature defined in~\cite[\S2]{BV}; it is
closed with respect to the covariant Cartan differential $d_\Theta -\iota
_\xi $.

  \begin{proof}
 The corresponding expression in the Weil model is 
  \begin{equation}\label{eq:8}
     \Omega -\theta ^a\wedge \iota _a\Omega + \frac 12\theta ^a\wedge \theta
     ^b\cdot\iota _b\iota _a\Omega - \nu ^a\cdot\iota _a\Theta . 
  \end{equation}
The map $M\to\BNG$ given by $Q\to M$ with connection~$\phi $ induces a
pullback on de Rham complexes known as the \emph{Chern-Weil
homomorphism}~\cite{MQ}.  It sends~\eqref{eq:8} to the curvature~$\Omega _Q$
of the induced connection on~$\pi _G$, by Proposition~\ref{thm:1}, which
proves~\eqref{eq:7} since forms are determined by their pullbacks to test
manifolds.
  \end{proof}

  \subsection*{Equivariant families of Dirac operators}

We translate the main construction~\eqref{eq:1} from principal bundles to
vector bundles.  Let $q\:Y\to T$ be a principal $G$-bundle with
connection~$U\subset TY$, and suppose $E\to Y$ is a vector bundle with a
$G$-action and $G$-invariant covariant derivative~$\nabla $.  Define
$q_*\nabla $ on $E/G\to T$ by
  \begin{equation}\label{eq:12}
     (q_*\nabla )\mstrut _\xi s = \nabla _{\hx}(q^*s), 
  \end{equation}
where $\hx$~is the horizontal lift of the tangent vector~$\xi $ and $s$~is a
section of $E/G\to T$.  Given $\pi \:E\to X$ with $G$-action and
$G$-invariant~$\nabla $, and a $G$-bundle $Q\to M$ with connection~$V$,
apply~\eqref{eq:12} to $Y=Q\times X$, $U=V\oplus TX$ to construct ~$\nabla
_Q$ on $E_Q\to X_Q$ (i.e., $Q\times \mstrut _GE\to Q\times \mstrut _{G}X$).

Recall that if $E\to Y$ is a vector bundle with covariant derivative~$\nabla
$ and $\xymatrix@1@C=16pt{E'\;\,
\ar@{^{(}->}@<.3ex>[r]^i&\;E\ar@{->>}@<.5ex>[l]^p}$ is a complemented
subbundle, then there is a \emph{compressed} covariant derivative
$\comp{E'}{E}\nabla =p\circ \nabla \circ i$ on~$E'\to Y$.  The compression is
transitive for iterated complemented subbundles $\xymatrix@1@C=16pt{E''\;\,
\ar@{^{(}->}@<.3ex>[r]&\;E'\ar@{->>}@<.5ex>[l] \;\,
\ar@{^{(}->}@<.3ex>[r]&\;E\ar@{->>}@<.5ex>[l] }$.

Let $\pi \:X\to S$ be a smooth fiber bundle.  A \emph{relative Riemannian
structure} is a metric~$g^{X/S}$ on $T(X/S)\to X$ together with a horizontal
distribution~$W\subset TX$.  It determines a relative Levi-Civita covariant
derivative~$\nabla ^{X/S}$ on the relative tangent bundle $T(X/S)\to X$ as
follows.  A Riemannian metric~$g^S$ on~$S$ induces a Riemannian metric~$g^X$
on~$X$ which makes~$\pi $ a Riemannian submersion.  Let $\nabla ^X$~be the
Levi-Civita covariant derivative on~$X$.  Then $\nabla
^{X/S}=\comp{T(X/S)}{TX}\nabla ^X$ is independent of~$g^S$.  If the fibers
of~$\pi $ are closed manifolds, and if there is a relative spin structure,
then there is an associated family of Dirac operators.  Quillen~\cite{Q}
constructed a metric on the determinant line bundle $\Det\to S$; it carries a
compatible covariant derivative whose curvature is the 2-form
component of the pushforward of the $\Ahat$~polynomial applied to the
relative curvature~\cite{BF}:
  \begin{equation}\label{eq:9}
     \omega = 2\pi i\left[\int_{X/S}\Ahat(\Omega ^{X/S})\right]_{(2)}
  \end{equation}
There is an extension for families of generalized Dirac operators.
 
Now suppose a Lie group~$G$ acts on $X\to S$ preserving all the data.  (The
``preservation'' of the relative spin structure is additional data.)  Then
there is an induced $G$-action on $\Det\to S$ which preserves the covariant
derivative, and so an equivariant curvature 
  \begin{equation}\label{eq:10}
     \omega _G = \omega - \mu 
  \end{equation}
in the Cartan model, where $\mu \in \Hom(\mathfrak{g},\Omega ^0_S)$ is the
moment map.

  \begin{theorem}[]\label{thm:4}
   \begin{equation}\label{eq:11}
     \omega _G= 2\pi i\left[\int_{X/S}\Ahat(\Omega ^{X/S}_G)\right]_{(2)}
  \end{equation}
  \end{theorem}

\noindent
 The equivariant curvature~$\Omega _G^{X/S}$ is~\eqref{eq:7}, and the
integrand is the $G$-equivariant Chern-Weil form associated to the
$\Ahat$~polynomial.  To compute the moment map it suffices to take~$S=\pt$.

  \begin{lemma}[]\label{thm:6}
 Fix a Lie group $G$ and a principal $G$-bundle $Q\to M$ with connection.

 \begin{enumerate}[label=\textnormal{(\roman*)}]

 \item Let $X$~be a Riemannian manifold with Levi-Civita covariant
derivative~$\nabla $.  Suppose $G$~ acts on~$X$ by isometries.  Then
$\nabla _Q$ is the relative Levi-Civita covariant derivative on $\XQ\to M$.

 \item Let $\pi \:X\to S$ be a fiber bundle with relative Riemannian
structure~$\bigl(g^{X/S},W\subset TX\bigr)$; there is an induced relative
Levi-Civita covariant derivative~$\nabla ^{X/S}$.  If $G$~acts preserving all
data, then $\nabla ^{X/S}_Q$ is the relative Levi-Civita covariant derivative
on $\XQ\to\SQ$.

 \end{enumerate}
  \end{lemma}

  \begin{proof}
 Fix a metric~$g^M$ and so an induced~$g^Q$ such that $Q\to M$ is a
Riemannian submersion.  Then $q\:Q\times X\to \XQ$ is a Riemannian submersion
with horizontal distribution $V\oplus TX$.  It is also a principal
$G$-bundle, and that distribution is a connection.  Then the Levi-Civita
covariant derivatives satisfy $\nabla ^{\XQ}=q_*\,\comp{V\oplus TX}{TQ\oplus
TX}\, \nabla ^{Q\times X}$, as follows by checking that the right hand side
preserves the metric and is torsionfree.  Hence the relative Levi-Civita
covariant derivative on $\XQ\to M$ is
  \begin{equation}\label{eq:14}
     \begin{aligned} \nabla ^{\XQ/M} &=
      \comp{T(\XQ/M)}{T\XQ}\,q_*\,\comp{V\oplus TX}{TQ\oplus TX}\, \nabla
      ^{Q\times X} \\[4pt] &= q_*\,\comp{TX}{V\oplus TX}\,\comp{V\oplus
      TX}{TQ\oplus TX}\, \nabla ^{Q\times X} \\[4pt] &= q_*\,\comp{TX}{TQ\oplus
      TX}\, \nabla ^{Q\times X} \\[4pt] &= q_*\,\nabla =\nabla _Q.\end{aligned} 
  \end{equation}
 
For~(ii) consider the iterated fiber bundle $Q\times
X\xrightarrow{\;q\;}\XQ\to \SQ\to M$.  Then using~\eqref{eq:12},
\eqref{eq:14}, and the transitivity of compression, we have
  \begin{equation}\label{eq:15}
     \nabla _Q^{X/S} = q_*\nabla ^{X/S}= q_*\,\comp{T(X/S)}{TX}\nabla =
     \comp{T(\XQ/\SQ)}{T(\XQ/M)}\,q_*\nabla =
     \comp{T(\XQ/\SQ)}{T(\XQ/M)}\,\nabla ^{\XQ/M}= \nabla ^{\XQ/\SQ}. 
  \end{equation}
  \end{proof}

  \begin{proof}[Proof of Theorem~\ref{thm:4}]
 The differential Borel quotient $\XG\to\SG$ inherits the relative Riemannian
metric and, by our basic construction, a horizontal distribution.  The
relative Levi-Civita covariant derivative, determined by its pullbacks~
$\nabla ^{X/S}_Q$ induced by maps $M\to\BNG$, is $\nabla ^{X/S}_G$; this is
Lemma~\ref{thm:6}(ii).  Therefore, its curvature is~$\Omega _G^{X/S}$.  Now
\eqref{eq:11}~follows by testing against maps $M\to\BNG$ and
applying~\eqref{eq:9} to the pullback fiber bundles of smooth manifolds.
  \end{proof}

  \begin{remark}[]\label{thm:5}
 The ``covariant anomaly formula'' (e.g. \cite[(3.60)]{BZ},
\cite[(3.57)]{AgG}) in quantum field theory is the moment map~$\mu $.  In the
cited formulas the base~$S$ is the space of covariant derivatives on a fixed
vector bundle $E\to Y$, and $X=S\times Y$.  The group~$G$ is the infinite
dimensional group of gauge transformations and the anomaly is the obstruction
to descending $\Det\to S$ to the quotient by~$G$, i.e., the moment map
precisely~$\mu $.  Theorem~\ref{thm:4} applies by restricting to finite
dimensional subgroups of~$G$.  Fix a connection in~$S$ with curvature~$F$,
and fix an infinitesimal gauge transformation~$v$.  The equivariant
curvature~\eqref{eq:7} evaluated on~$v$ is $F-v$, and in the
formula~\eqref{eq:11} we use the Chern character in place of the
$\Ahat$~polynomial.  Assuming $\dim X=2n$ we find $\mu = -2\pi i
\left(\frac{i}{2\pi }\right)^n\frac{1}{n!}\int_{X} \tr \,vF^n$, which agrees
with \cite[(3.57)]{AgG} (up to~$2\pi i$).  There is a similar formula for the
``gravitational'' case, in which $S$~is the space of Riemannian metrics on a
fixed manifold~$Y$ and $G$~is the group of diffeomorphisms
(\cite[(5.19)]{BZ}, \cite[(5.32)]{AgG}).  Mathematical treatments of
anomalies often first descend to the quotient~$S/G$, essentially using the
construction ~\eqref{eq:12}, in which case the anomaly is the determinant
line bundle with covariant derivative associated to the family of Dirac
operators over that quotient.  The precise relationship between the moment
map for the $G$-action on the determinant bundle over~$S$ and the determinant
bundle over~$S/G$ ties together different approaches to anomalies.
  \end{remark}

\vskip -40pt\ 

\providecommand{\bysame}{\leavevmode\hbox to3em{\hrulefill}\thinspace}
\providecommand{\MR}{\relax\ifhmode\unskip\space\fi MR }
\providecommand{\MRhref}[2]{%
  \href{http://www.ams.org/mathscinet-getitem?mr=#1}{#2}
}
\providecommand{\href}[2]{#2}

  \end{document}